\documentclass[letterpaper,10pt, journal,twoside]{IEEEtran}  

\IEEEoverridecommandlockouts                              

\usepackage{amsmath,amssymb,color,cite}
\usepackage{graphicx}

\usepackage{latexsym}
\usepackage{verbatim}
\usepackage{cite}
\usepackage[english]{babel}
\usepackage[latin1]{inputenc}
\usepackage{enumerate}
\usepackage{amsmath,amsthm}

\usepackage{tikz}
\usetikzlibrary{tikzmark}
\usepackage{amsfonts}
\usepackage{amssymb}
\usepackage{amsbsy}
\usepackage{bm}

    \let\geq\geqslant

\newcommand{\bmu}{\bm{u}}
\newcommand{\bmx}{\bm{x}}

\newcommand{\bmw}{\bm{w}}
\newcommand{\bmy}{\bm{y}}

\DeclareMathOperator{\im}{im}

\DeclareMathOperator{\rank}{rank}

\let\geq\geqslant

\newcommand{\calA}{\ensuremath{\mathcal{A}}}

\newcommand{\calJ}{\ensuremath{\mathcal{J}}}

\newcommand{\calL}{\ensuremath{\mathcal{L}}}

\newcommand{\calN}{\ensuremath{\mathcal{N}}}

\newcommand{\calP}{\ensuremath{\mathcal{P}}}

\newcommand{\calV}{\ensuremath{\mathcal{V}}}

\newcommand{\bmat}{\begin{matrix}}
\newcommand{\emat}{\end{matrix}}
\newcommand{\bbm}{\begin{bmatrix}}
\newcommand{\ebm}{\end{bmatrix}}
\newcommand{\bbma}{\begin{bmatrix*}[r]}
\newcommand{\ebma}{\end{bmatrix*}}
\newcommand{\bpm}{\begin{pmatrix}}
\newcommand{\epm}{\end{pmatrix}}
\newcommand{\bvm}{\begin{vmatrix}}
\newcommand{\evm}{\end{vmatrix}}
\newcommand{\bse}{\begin{subequations}}
\newcommand{\ese}{\end{subequations}}
\newcommand{\beq}{\begin{equation}}
\newcommand{\eeq}{\end{equation}}
\newcommand{\ben}{\renewcommand{\labelenumi}{\arabic{enumi}.}
\renewcommand{\theenumi}{\arabic{enumi}}\begin{enumerate}}

\newcommand{\een}{\end{enumerate}}

\newcommand{\beni}{\renewcommand{\labelenumi}{\roman{enumi}.}
\renewcommand{\theenumi}{\roman{enumi}}\begin{enumerate}}

\newcommand{\eeni}{\end{enumerate}}

\newcommand{\bena}{\renewcommand{\labelenumi}{\alph{enumi}.}
\renewcommand{\theenumi}{\alph{enumi}}\begin{enumerate}}

\newcommand{\eena}{\end{enumerate}}

\newcommand{\bit}{\begin{itemize}}
\newcommand{\eit}{\end{itemize}}
\newcommand{\bthe}{\begin{theorem}}
\newcommand{\ethe}{\end{theorem}}
\newcommand{\blem}{\begin{lemma}}
\newcommand{\elem}{\end{lemma}}
\newcommand{\bprop}{\begin{proposition}}
\newcommand{\eprop}{\end{proposition}}
\newcommand{\bex}{\begin{example}}
\newcommand{\eex}{\end{example}}
\newcommand{\bas}{\begin{assumption}}
\newcommand{\eas}{\end{assumption}}
\newcommand{\bre}{\begin{remark}}
\newcommand{\ere}{\end{remark}}
\newcommand{\bcor}{\begin{corollary}}
\newcommand{\ecor}{\end{corollary}}
\newcommand{\bdfn}{\begin{definition}}
\newcommand{\edfn}{\end{definition}}
\newcommand{\bcon}{\begin{conjecture}}
\newcommand{\econ}{\end{conjecture}}

\newcommand{\pset}[1]{\ensuremath{\{#1\}}}

\newcommand{\zset}{\ensuremath{\pset{0}}}

\newcommand{\qand}{\quad\text{ and }\quad}

\newcounter{todocounter}
\usepackage[colorinlistoftodos]{todonotes}

\newtheorem{theorem}{Theorem}
\newtheorem{lemma}[theorem]{Lemma}
\newtheorem{cor}[theorem]{Corollary}
\newtheorem{prop}[theorem]{Proposition}
\theoremstyle{definition}
\newtheorem{definition}[theorem]{Definition}
\newtheorem{example}[theorem]{Example}

\newtheorem{conjecture}[theorem]{Conjecture}
\newtheorem{remark}[theorem]{Remark}

\title{Informativity of noisy data for structural properties of linear systems} 
\author{Jaap Eising, Harry L. Trentelman~\IEEEmembership{Fellow,~IEEE}
	\thanks{The authors are with the Bernoulli Institute for Mathematics, Computer Science, and Artificial Intelligence, University of Groningen, Nij\-enborgh 9, 9747 AG, Groningen, The Netherlands. (email: {\footnotesize{\tt j.eising@rug.nl; h.l.trentelman@rug.nl}}).}
}

\begin{document}

\maketitle
\begin{abstract}
This paper deals with developing tests for checking whether an unknown system has certain structural properties. The tests that we are aiming at are in terms of noisy input-state-output data obtained from the unknown system. Since, in general, the data do not determine the unknown system uniquely, many systems are compatible with the same set of data. Therefore we can not apply system identification and apply existing, model based, tests. Instead, we will use the concept of informativity, and establish tests for informativity of the given noisy data. We will do this for a range of system properties, among which strong observability and detectability and strong controllability and stabilizability. These informativity tests will be in terms of rank tests on polynomial matrices that can be constructed from the noisy data. We will also set up a geometric framework for informativity analysis. Within that framework we will give geometric tests for informativity for strong observability, observability, and left-invertibilty. 
\end{abstract}

\section{Introduction}
In this paper we will study the problem of determining whether a given unknown dynamical system has  certain structural properties, based on noisy data obtained from this system.  One way to approach this problem is to use the data to identify an explicit model representing the system, and apply a suitable, model based, test to this model. In the present paper we will approach the problem from a different angle, and establish tests, directly on the noisy data, to check structural system properties. 

As a major tool, we will use the general framework of informativity of data, recently introduced in \cite{vanWaarde2020}. In that paper, it was shown that the data-driven approach can also be useful if the data do not give sufficient information to identify the `true' unknown system. In that case, a given set of data gives rise to a whole {\em family of system models}, all of which could have given the same data. On the basis of the data it is then impossible to distinguish between models, and a given system property will hold for the 'true' system model only if its holds for all models compatible with the data. Formalizing this, a set of data is called {\em informative} for a given system property if the property holds for all systems that could have given this set of data. 

In \cite{vanWaarde2020}, tests were established for checking whether a given set of noiseless data is informative for controllability, stability and stabilizability. In \cite{Eising2020}, informativity of noiseless data for observability was studied. 
In the present paper we will deal with informativity of noisy data.  We will establish informativity tests for several relevant structural system properties. More specifically, we will study informativity for observability and detectability, strong observability and detectability, strong controllability and stabilizability, and invertibility of linear systems. 
These structural properties are relevant in in a wide range of observer, filter and control design problems. For definitions and extensive treatments we refer to \cite{Morse1973, Silverman1968, Wonham1985,Schumacher1983,Hautus1983,Molinari1976}, and \cite{Trentelman2001} and the references therein.

Analysis of system properties based on data has been studied also in \cite{Wang2011,Niu2017,Zhou2018,Liu2014}, which deal with data-based controllability and observability analysis. Whereas in the present paper general data sets are allowed, these references impose restrictions on the data. The paper \cite{Park2009} deals with the problem of determining stability properties of input-output systems using time series data. More recently, there has also been an increasing interest in the problem of verifying dissipativity on the basis of system data. This problem has, for example, been addressed in \cite{Maupong2017b,Romer2019,Koch2020,Berberich2019}.

\section{Problem formulation}\label{sec:prob}
In this paper we will consider the linear discrete-time input-state-output system with noise given by 
\begin{subequations} \label{eq:full system}
	\begin{align} \bmx(t+1) &= A_{\rm true} \bmx(t)+B\bmu(t) + E\bmw(t), \label{eq:def sys ABCDE_a}\\ \bmy(t) &= C\bmx(t) + D\bmu(t) +F \bmw(t), \label{eq:def sys ABCDE_b}\end{align}
\end{subequations}
where $\bmx \in \mathbb{R}^n$ is the state, $\bmu \in \mathbb{R}^m$ a control input, $\bmy \in \mathbb{R}^p$ an output, and $\bmw \in \mathbb{R}^r$ is unknown noise.  We assume that $A_{\rm true}$ is an unknown matrix, but that the matrices $B$, $C$, $D$ and $E,F$ are known.  The assumption that
these matrices are known is reasonable, for example in networked systems, in which the input and output nodes are given, but the interconnection topology is unknown.  Typically, in that context, the matrices $B,C$ and $D$ would be matrices whose columns only contain $0$'s and $1$'s, with in each column at most one entry equal to 1. The term $E\bmw(t)$ represents process noise, whereas $F\bmw(t)$ represents measurement noise. The special case that $E =0$ and $F =0$ is called the {\em noiseless case}. 

We assume that we have input-state-output data concerning this unknown `true' system in the form of samples of $\bmx$, $\bmu$ and $\bmy$ on a given finite time interval $\{0,1 \ldots ,T\}$. These data are denoted by 
\begin{subequations}\label{eq:UXdata}
\begin{align}
U_-& := \bbm u(0) & u(1) & \cdots & u(T-1)\ebm, \\
X  & := \bbm x(0) & x(1) & \cdots & x(T)\ebm, \\
Y_-& := \bbm y(0) & y(1) & \cdots & y(T-1)\ebm. 
\label{eq: UXdata2}
\end{align}
\end{subequations}
It will be assumed that the data \eqref{eq:UXdata} are `harvested' from the true system \eqref{eq:full system}, meaning that there exists some matrix 
$$W_- = \bbm w(0) & w(1) & \cdots & w(T -1)\ebm$$ 
such that 
\begin{subequations}  \label{e:data}
\begin{align}
X_{+} & = A_{\rm true} X_{-} + B U_-  + E W_-, \\
Y_- & = C X_- + D U_- + FW_-, 
\end{align}
\end{subequations}
where we denote
\begin{align*}
X_{-}& := \bbm x(0) & x(2) & \cdots & x(T -1) \ebm, \\
X_{+}& := \bbm x(1) & x(2) & \cdots & x(T) \ebm.
\end{align*}	
We then say that the data are {\em compatible} with the true system $(A_{\rm true},B,C,D,E,F)$. 

The set of all $n \times n$ matrices $A$ such that the data \eqref{eq:UXdata} are compatible with the system $(A,B,C,D,E,F)$ is denoted by $\calA_{\rm dat}$, i.e.,
\begin{align} 
\calA_{\rm dat} & := \big\{ A \in \mathbb{R}^{n \times n}  \mid  
\exists W_- :  \nonumber \\ &\begin{pmatrix} X_{+}\\ Y_- \end{pmatrix}  = \begin{pmatrix} A & B \\ C  &  D \end{pmatrix} \begin{pmatrix} X_{-} \\ U_- \end{pmatrix} 
+ \begin{pmatrix} E  \\ F \end{pmatrix} W_-  \big\}. \label{eq:compat}
\end{align}
Let $\calP$ denote some system theoretic property that might or might not hold for a given linear system. The general problem that we will address in this paper is to determine from the data harvested from \eqref{eq:full system} whether the property $\cal P$ holds for the unknown true system $(A_{\rm true},B,C,D,E,F)$. Since on the basis of the data we can not distinguish between the true $A_{\rm true}$ and any $A \in  \calA_{\rm dat}$, we need to check whether the property holds for {\em all} systems $(A,B,C,D,E,F)$ with $A \in  \calA_{\rm dat}$. Following \cite{vanWaarde2020}, in that case we call the data {\em informative for property $\calP$.}
\begin{example}
For $\cal P$ take the property `$(A,B)$ is a controllable pair'. Suppose that on the basis of the data $(U_-,X, Y_-)$ we want to determine whether $\calP$ holds for the pair $(A_{\rm true},B)$ corresponding to the true system. This requires to check whether the data are informative for property $\calP$. Using Theorem 8 in \cite{vanWaarde2020}, it can be shown that in the noiseless case (i.e. the case that $E =0$ and $F =0$) the data $(U_-,X, Y_-)$ are informative for $\calP$ if and only if $\rank \bbm X_+ - \lambda X_- & B \ebm = n$ for all $\lambda \in \mathbb{C}$. 
\end{example}
\begin{example}
For $\calP$ take the property `the pair $(C,A)$ is detectable'. In the noiseless case it can be shown that the data $(U_-,X, Y_-)$ are informative for $\calP$ if and only if $\ker C \subseteq \im X_-$ and for all $\lambda \in\mathbb{C}$ with $|\lambda| \geq 1$ we have
$$
	\rank \begin{pmatrix} X_+ - BU_- - \lambda X_-  \\ CX_- \end{pmatrix} =  \rank X_- .
$$  
This will be one of the results in this paper.
\end{example} 
\begin{remark} \label{rem:independent}
We note that the case of independent process noise and measurement noise is also covered by the noisy model \eqref{eq:full system} introduced above. The noise matrices should then be taken of the form $E = (E_1~ 0)$ and $F = (0~ F_2)$, while the noise signal is given by the vector  $\bmw = \begin{pmatrix} \bmw_1 \\ \bmw_2 \end{pmatrix}$ and likewise $W_- = \begin{pmatrix} W_{1 -} \\ W_{2 -} \end{pmatrix}$. A special case of this is that only process noise occurs, in which case $F_2 $ is void and $E = E_1$ and $F = 0$.
In other words, in the case of independent process and measurement noise we have $A \in \calA_{\rm dat}$ if and only if there exists a matrix $W_{1-}$ such that $X_+ = A X_- + B U_- + E_1 W_{1-}$. The equation $Y_- = CX_- + D U_- + F_2 W_{2-}$ can then be ignored since it does not put any constraint on $A$. 
\end{remark}
The purpose of this paper is to establish necessary and sufficient conditions on the input-state-output data obtained from \eqref{eq:full system} to be informative for a range of system properties $\calP$. Throughout, we will restrict ourselves to the situation introduced above, namely, that the state map $A_{\rm true}$ is unknown, but that the matrices $B,C$ and $D$ are known. We will study both the noisy case as well as the noiseless case. In the noisy case it will be assumed that the noise matrices $E$ and $F$ are known. 

The outline of the remainder of this paper is as follows. In Section \ref{sec:rank property}, we will state and prove a theorem that will be instrumental in order to obtain our results on informativity in the rest of the paper. The theorem expresses a rank property of the Rosenbrock system matrix of the unknown system in terms of a polynomial matrix that collects available information about the unknown system. In Section \ref{sec:informativity analysis}, this result will be applied to obtain necessary and sufficient conditions for informativity of noisy data for the following system properties:
\begin{itemize}
\item strong observability and strong detectability of\\ $(A,B,C,D)$,
\item observability and detectability of $(C,A)$,
\item strong controllability and strong stabilizability of \\$(A,B,C,D)$,
\item controllability and stabilizability of $(A,B)$.
\end{itemize}
In Section \ref{sec:geometric}, we apply ideas from the geometric approach to linear systems, see \cite{Trentelman2001, Wonham1985} to set up a geometric framework for informativity analysis for strong observability and obser\-vability. This framework will then be applied to the analysis of informativity for left-invertibility. Finally, in Section \ref{sec:conclusions} we close this paper with concluding remarks.

\section{A rank property for an affine set of systems} \label{sec:rank property}
In this section we will establish a general framework that will enable us characterize informativity of input-state-output data for the properties listed in Section II.
 
Let $P \in \mathbb{R}^{n \times r}$, $Q\in \mathbb{R}^{\ell \times n}$ and $R\in \mathbb{R}^{\ell \times r}$ be given matrices. Here, $r$ and $\ell$ are positive integers, and the symbol $n$ has the usual meaning of state space dimension. Using these matrices, we define an affine space of state matrices $A$ by
\begin{equation}\label{eq:def sig a} 
\calA := \{ A \in \mathbb{R}^{n\times n} \mid R = Q A P \}.
 \end{equation}
It is easily seen that $\calA$ is nonempty if and only if $\im R \subseteq \im Q$ and $\ker P \subseteq \ker R$. Assume this to be the case.

Now let $B \in \mathbb{R}^{n \times m}$, $C\in \mathbb{R}^{p \times n}$ and $D \in \mathbb{R}^{p \times m}$ be given, and for each  $A\in\calA$ consider the system
\begin{subequations}\label{eq:def sys ABCD}
	\begin{align} \bmx(t+1) &= A\bmx(t)+B\bmu(t), \\ \bmy(t) &= C\bmx(t) + D\bmu(t). \end{align}
\end{subequations}
The Rosenbrock system matrix associated with the system \eqref{eq:def sys ABCD} is defined as the polynomial matrix
\begin{equation}  \label{eq:Rosenbrock}
\begin{pmatrix} A- sI & B \\ C & D \end{pmatrix}.
 \end{equation} 
In addition, we we will consider the polynomial matrix 
\begin{equation} \label{eq:XYZ}
\begin{pmatrix} R- sQP & QB\\ CP &D \end{pmatrix}
\end{equation}
associated with the given matrices $(P,Q,R)$ and $(B,C,D)$.
The following theorem expresses a uniform rank property of the set of system matrices \eqref{eq:Rosenbrock}, with $A$ ranging over the affine set $\calA$, in terms of a rank property of the single polynomial matrix \eqref{eq:XYZ}.
\begin{theorem}\label{thm:rank equivalences}
	Let $(P,Q,R)$ and $(B,C,D)$ be given. Then 
	\begin{equation}\label{eq:rank equivalences 1} \rank  \begin{pmatrix} A- \lambda I & B \\ C & D \end{pmatrix} =n+ \rank\begin{pmatrix} B \\D \end{pmatrix}
	\end{equation}
	for all $A\in\calA$ and $\lambda \in\mathbb{C}$ if and only if $C^{-1}\im D\subseteq \im P$ and 	\begin{equation}\label{eq:rank equivalences 2}
	\rank \begin{pmatrix} R- \lambda QP & QB\\ CP &D \end{pmatrix} = \rank P + \rank \begin{pmatrix} QB \\D \end{pmatrix}
	\end{equation} 
	for all $\lambda \in\mathbb{C}$ .\footnote{For a given subspace $\calL$ and matrix $M$ we denote by $M ^{-1} \calL$ the inverse image $\{x \mid Mx \in \calL\}$.}
	
	In addition, \eqref{eq:rank equivalences 1} holds for all $A\in\calA$ and $\lambda\in\mathbb{C}$ such that $|\lambda|\geq 1$ if and only if  $C^{-1}\im D\subseteq \im P$ and \eqref{eq:rank equivalences 2} holds for all $\lambda\in\mathbb{C}$ such that $|\lambda|\geq 1$.
\end{theorem}
\begin{proof} To start the proof, first observe that for any $A\in \calA$: 
\begin{equation} \label{eq:useful}
 \begin{pmatrix} A- \lambda I & B \\ C & D \end{pmatrix} = \begin{pmatrix} A-\lambda I & I & 0\\ C& 0 & I \end{pmatrix} \begin{pmatrix} I & 0 \\ 0 & B \\ 0 & D \end{pmatrix}.
\end{equation}
Note that for any pair of matrices $M$ and $N$ we have $\rank MN=\rank N$ if and only if $\ker MN = \ker N$. By applying this to \eqref{eq:useful}, we see that \eqref{eq:rank equivalences 1} is equivalent to
\[ 
 \begin{pmatrix} A- \lambda I & B \\ C & D \end{pmatrix} \begin{pmatrix}\xi \\ \eta \end{pmatrix} = 0 \implies  \begin{pmatrix} I & 0 \\ 0 & B \\ 0 & D \end{pmatrix}\begin{pmatrix}\xi \\ \eta \end{pmatrix} = 0 .
\]
It is straightforward to check that, in turn, this holds if and only if
\begin{equation} \label{eq: equiv to rank 1} 
 \begin{pmatrix} A- \lambda I & B \\ C & D \end{pmatrix}\begin{pmatrix}\xi \\ \eta \end{pmatrix} = 0 \implies  \xi =0 .
\end{equation}
Similarly, note that for all $A \in \calA$
\[ 
 \begin{pmatrix} R- \lambda QP & QB\\ CP &D \end{pmatrix}\! = \!\begin{pmatrix} Q(A-\lambda I) \!\! & I \! & 0\\ C\!\! & 0\! & I \end{pmatrix} \begin{pmatrix} P\!\! & 0 \\ 0\!\! & QB \\ 0\!\! & D \end{pmatrix}\!\!.
\]
This makes \eqref{eq:rank equivalences 2} equivalent to 
\begin{equation}\label{eq: equiv to rank 2}
 \begin{pmatrix} R- \lambda QP & QB\\ CP &D \end{pmatrix} \begin{pmatrix} \nu \\ \eta \end{pmatrix}=0 \! \implies P\nu = 0  
\end{equation}
From here on, we will prove the first statement of the theorem, noting any changes required for the second part. 

$(\Leftarrow)$: Let $A\in\calA$ and $\lambda\in\mathbb{C}$ (resp. $\lambda\in\mathbb{C}$ such that $|\lambda|\geq 1$). Assume that $C^{-1}\im D \subseteq \im P$ and \eqref{eq:rank equivalences 2} holds for $\lambda$. We will prove that \eqref{eq: equiv to rank 1} holds. For this, let $\xi$ and $\eta$ satisfy 
\[ 
 \begin{pmatrix} A- \lambda I & B \\ C & D \end{pmatrix}\begin{pmatrix}\xi \\ \eta \end{pmatrix} = 0. 
\]
Since $\xi\in C^{-1}\im D \subseteq \im P$, we can write $\xi= P \nu$ for some $\nu$. Now, by pre-multiplying with $\begin{pmatrix} Q & 0 \\ 0& I \end{pmatrix}$ we obtain that
\[
 \begin{pmatrix} R- \lambda QP & QB\\ CP &D \end{pmatrix} \begin{pmatrix}\nu \\ \eta \end{pmatrix}=0.
\]
We can now apply \eqref{eq: equiv to rank 2} and thereby conclude that $\xi=P\nu =0$. This proves that \eqref{eq: equiv to rank 1} holds. 

$(\Rightarrow)$: Assume that \eqref{eq: equiv to rank 1} holds for all $A\in\calA$ and $\lambda\in\mathbb{C}$ (resp. $\lambda\in\mathbb{C}$ such that $|\lambda|\geq 1$). We will first prove that $C^{-1}\im D \subseteq \im P$. 

Let $\hat{x}\in C^{-1}\im D\setminus \im P$, that is, $\hat{x}\not\in \im P$ and there exists a $\hat{u}$ such that $C\hat{x} + D\hat{u} = 0$. Without loss of generality take $\hat{x}$ and $\hat{u}$ as real vectors. Take any $A\in \calA$ and $\mu\in\mathbb{R}$ (resp. $\mu\in\mathbb{R}$ such that $|\mu|\geq 1$). Let $A_0$ be any real $n \times n$ matrix such that $Q A_0 P = 0 $ and $A_0\hat{x}= -(A-\mu I)\hat{x}-B\hat{u}$. Note that such matrix exists as $\hat{x}\not\in \im P$ and $-(A-\mu I)\hat{x}-B\hat{u}$ is a real vector. Now define $\bar{A}:=A+A_0$. Note that $\bar{A}\in\calA$ and 
\[ \begin{pmatrix} \bar{A}-\mu I & B \\ C & D \end{pmatrix}\begin{pmatrix}\hat{x} \\ \hat{u} \end{pmatrix} = 0. \]
By \eqref{eq: equiv to rank 1}, we see that $\hat{x}=0$, which contradicts with $x\not\in\im P$. Therefore $C^{-1}\im D \subseteq \im P$.

We now move to proving \eqref{eq: equiv to rank 2}. Let $\lambda\in\mathbb{C}$ (resp. $\lambda\in\mathbb{C}$ such that $|\lambda|\geq 1$), and let $\nu$ and $\eta$ satisfy 
\[\begin{pmatrix} R-\lambda QP \!\!& QB \\ CP\!\! & D \end{pmatrix} \begin{pmatrix}\nu \\ \eta \end{pmatrix}=0.\]
Denote $\xi = P\nu$, then we see that $C\xi +D\eta=0$ and $(A-\lambda I)\xi+B\eta\in\ker Q$ for any $A\in \calA$.

We will prove \eqref{eq: equiv to rank 2} holds in three separate cases: First, we prove the statement for real $\lambda$. For complex $\lambda$ we consider the cases where the real and complex parts of $\xi$ are linearly dependent and where these are linearly independent.

First suppose that $\lambda \in \mathbb{R}$. Then, without loss of generality, $\nu$ and $\eta$ are real, and as such $\xi$ is real. Suppose that $\xi\neq 0$, and take any $A\in \calA$. Let $A_0$ be any real $n \times n$ matrix such that $A_0 \xi=-(A-\lambda I)\xi-B\eta$ and $QA_0P=0$. Such a matrix exists as $-(A-\lambda I)\xi-B\eta \in \ker Q$ and is a real vector and $\xi\neq 0$. Now take $\bar{A}=A+A_0$. Then it is immediate that $\bar{A}\in \calA$ and:
\[ \begin{pmatrix} \bar{A}-\lambda I & B \\ C & D \end{pmatrix}\begin{pmatrix}\xi \\ \eta \end{pmatrix} = 0. \]
As \eqref{eq: equiv to rank 1} holds for $\bar{A}$ by assumption, we see that $\xi=0$, which leads to a contradiction. Therefore $\xi=0$.

Now consider that case where $\lambda\not\in\mathbb{R}$. Suppose that the real and complex parts of $\xi$ are linearly dependent. Therefore, there exist real scalars $\alpha,\beta\in \mathbb{R}$ and a real vector $r$ such that $\xi = (\alpha +i\beta)r$. Let $\hat{r} = (\alpha-i\beta) \xi = (\alpha^2+\beta^2)r$. Let $A\in\calA$, then: 
\[\begin{pmatrix} Q(A-\lambda I) \!\!& QB \\ C\!\! & D \end{pmatrix} \begin{pmatrix}\hat{r} \\ (\alpha-i\beta)\eta \end{pmatrix}=0.\]
Denote $\lambda = a+bi$, where $b\neq 0$, and $(\alpha-i\beta)\eta= \eta_1+i\eta_2$. Then we see that: 
$
Q(A-aI)\hat{r} +QB\eta_1  = -b\hat{r} + QB\eta_2  =0$ and $C\hat{r} +D\eta_1 = D\eta_2 =0$. 
Let $\mu\in\mathbb{R}$ (resp. $\mu\in\mathbb{R}$ such that $|\mu|\geq 1$). Note that 
\[\begin{pmatrix} Q(A-\mu I) \!\!& QB \\ C\!\! & D \end{pmatrix} \begin{pmatrix}b\hat{r} \\ b\eta_1 +(\mu-a)\eta_2 \end{pmatrix}=0.\]
As $\mu$ is real, we can now apply the previous part of the proof to note that $b\hat{r}=0$, which holds only if $\xi=0$.

Now suppose that $\xi = Pp+iPq$, where $Pp$ and $Pq$ are linearly independent. If we take any $A\in\calA$, we know that $Q(A-\lambda I)\xi+QB\eta=0$, and that we can denote $(A-\lambda I)\xi+B\eta= \zeta_1+\zeta_2i$, where $\zeta_1,\zeta_2\in\ker Q$. Take $A_0$ any real map such that $A_0 Pp= -\zeta_1$, $A_0Xq=-\zeta_2$ and $QA_0P =0$. Such a map exists as $Pp$ and $Pq$ are linearly independent. Now take $\bar{A}=A+A_0$, then $\bar{A}\in \calA$ and clearly 
\[ \begin{pmatrix} \bar{A}-\lambda I & B \\ C & D \end{pmatrix}\begin{pmatrix}\xi \\ \eta \end{pmatrix} = 0. \]
Using \eqref{eq: equiv to rank 1}, this implies that $\xi=0$. This is a contradiction with the fact that $Pp$ and $Pq$ are linearly independent. 
\end{proof}

\section{Data driven informativity analysis} \label{sec:informativity analysis}
In this section we will apply Theorem \ref{thm:rank equivalences} to obtain necessary and sufficient conditions for informativity of input-state-output data for the system properties listed in Section II. For a given system \eqref{eq:def sys ABCD}
we will denote by $\bmx(t, x_0,\bmu)$ and $\bmy(t, x_0,\bmu)$ the state and output sequence corresponding to the initial state $\bmx(0)=x_0$ and input sequence $\bmu$. 

\subsection{Informativity for strong observability and detectability}
We first briefly review the properties of strong observability and strong detectability (see also \cite{Trentelman2001}).
\begin{definition}
The system \eqref{eq:def sys ABCD} is called \textit{strongly observable} if for each $x_0\in\mathbb{R}^n$ and  input sequence $\bmu$ the following holds: $\bmy(t, x_0, \bmu)=0$ for all $t \in \mathbb{Z}_+$ implies that $x_0=0$. The system is called \textit{strongly detectable} if for all $x_0\in\mathbb{R}^n$ and every input sequence $\bmu$ the following holds: $\bmy(t, x_0,\bmu)=0$ for all $t \in \mathbb{Z}_+$ implies that $\lim_{t\rightarrow \infty} \bmx(t, x_0,\bmu) = 0$.
\end{definition}
%
For continuous-time systems, 
necessary and sufficient conditions for strong observability and strong detectability were formulated in \cite{Trentelman2001}.  It can be verified that also the discrete-time system \eqref{eq:def sys ABCD} is strongly observable (strongly detectable) if and only if the pair $(C + DK, A+BK)$ is observable (detectable) for all $K$. It is also straightforward to verify the following.
\begin{prop}  \label{prop:str obs iff rank}
	The system \eqref{eq:def sys ABCD} is strongly observable if and only if for all $\lambda\in\mathbb{C}$
	\begin{equation}\label{eq:haut for str obs}\rank \begin{pmatrix} A-\lambda I & B \\ C & D \end{pmatrix} =n+ \rank\begin{pmatrix} B \\D \end{pmatrix}. \end{equation}
	The system  \eqref{eq:def sys ABCD} is strongly detectable if and only if \eqref{eq:haut for str obs} holds for all $\lambda\in\mathbb{C}$ such that $|\lambda|\geq 1$.
\end{prop}
As in Section III, we now consider the situation that only the matrices $B,C$ and $D$ are given, and that the matrix $A$ can be any matrix from the affine set \eqref{eq:def sig a}
with $P,Q$ and $R$ given matrices. By applying Theorem \ref{thm:rank equivalences} we then get the following necessary and sufficient conditions for strong observability and strong detectability of {\em all} systems \eqref{eq:def sys ABCD} with $A$ ranging over the affine set $\calA$.

\begin{theorem}[Uniform rank condition]\label{thm:str obs hautus}
	Let $(P,Q,R)$ and $(B,C,D)$ be given matrices.  Then \eqref{eq:def sys ABCD} is strongly observable for all $A\in\calA$ if and only if $C^{-1}\im D\subseteq \im X$ and for all $\lambda \in\mathbb{C}$ we have
	\begin{equation}\label{eq:hautus str obs xyz}
	\rank\begin{pmatrix} R-\lambda QP & QB\\ CP &D \end{pmatrix} = \rank P + \rank \begin{pmatrix} QB \\D \end{pmatrix}.
	\end{equation}  
	Similarly,  \eqref{eq:def sys ABCD} is strongly detectable for all $A\in\calA$ if and only if $C^{-1}\im D\subseteq \im P$ and \eqref{eq:hautus str obs xyz} holds for all $\lambda\in\mathbb{C}$ such that $|\lambda|\geq 1$.
\end{theorem}
\begin{proof} This follows immediately by combining Proposition~\ref{prop:str obs iff rank} and Theorem~\ref{thm:rank equivalences}. 
\end{proof}

We will now apply the previous result to informativity of input-state-output data. Suppose the data are $(U_-,X, Y_-)$. Recall Definition \eqref{eq:compat} of the affine set $\calA_{\rm dat}$ of all $n \times n$ matrices $A$ such that the data are compatible with the system $(A,B,C,D,E,F)$.
We want to obtain conditions under which the data are informative for strong observability and for strong detectability.  To this end, let $(M~~N)$  be any matrix such that 
\begin{equation} \label{eq:annihilator}
\ker (M~~N) = \im \begin{pmatrix} E \\ F \end{pmatrix}.
\end{equation}
Then we have $A \in \calA_{\rm dat}$ if and only if $R= MAX_-$ with 
\begin{equation} \label{eq:Z}
R:= (M~ ~N) \begin{pmatrix} X_+ - BU_- \\ Y_- -CX_- -DU_- \end{pmatrix}.
\end{equation}
The following then immediately follows from Theorem \ref{thm:str obs hautus}.
\begin{theorem} \label{th:strong obs and det}
The data $(U_-,X,Y_-)$ are informative for strong observability if and only if 
$C^{-1}\im D\subseteq \im X_-$ and for all $\lambda \in\mathbb{C}$ we have
\begin{equation}\label{eq:hautus str obs xyz}
	\rank \begin{pmatrix} R - \lambda MX_-& MB \\ CX_- &D \end{pmatrix} =  \rank X_- + \rank \begin{pmatrix} MB \\D \end{pmatrix},
\end{equation} 
where $R$ is given by \eqref{eq:Z}. 
 
The data are informative for strong detectability if and only if $C^{-1}\im D\subseteq \im X_-$ and \eqref{eq:hautus str obs xyz} holds for all  $\lambda \in\mathbb{C}$ with $|\lambda| \geq 1$.
\end{theorem}
%
%
In the case of independent process and measurement noise (see Remark \ref{rem:independent}), in which $E = (E_1~ 0)$ and $F = (0~ F_2)$, we have $A \in \calA_{\rm dat}$ if and only if there exists a matrix $W_{1-}$ such that $X_+ = A X_- + B U_- + E_1 W_{1-}$. Thus, $A \in \calA_{\rm dat}$ if and only if 
$R= MAX_-$ with 
\begin{equation} \label{eq:newZ} 
R := M(X_+ - BU_-),
\end{equation}
and $M$ such that $\ker M = \im E_1 = \im E$. In this case, the formulation of Theorem \ref{th:strong obs and det} holds verbatim with this $M$, and the new $R$ given by \eqref{eq:newZ}.

Finally, for the special case $E =0$ (the case with no process noise), we have $A \in \calA_{\rm dat}$ if and only if $R= AX_-$ with 
\begin{equation} \label{eq:newnewZ} 
R := X_+ - BU_-~, 
\end{equation}
In that case, Theorem \ref{th:strong obs and det} holds verbatim with $M = I_n$ and $R$ given by \eqref{eq:newnewZ}.


\subsection{Informativity for observability and detectability}
Next, we turn to characterizing informativity of the data for the properties of observability and detectability. 
Consider the system
\begin{equation} \label{eq:our observed system}
\bmx(t+1) = A\bmx(t), ~ \bmy(t) = C\bmx(t). 
\end{equation}
The Hautus test states that \eqref{eq:our observed system} is observable (detectable) if and only if 
\[
\rank \begin{pmatrix} A - \lambda I \\
                                  C
          \end{pmatrix} = n
\]
for all $\lambda \in \mathbb{C}$ (for all  $\lambda \in \mathbb{C}$ with $|\lambda| \geq 1$).

Now, take the situation that only $C$ is known, that matrices $P,Q$ and $R$ are given, and that $A$ can be any matrix from the affine set $\calA$ given by \eqref{eq:def sig a}. By applying Theorem \ref{eq:Rosenbrock} to the special case $B =0$ and $D =0$, we then obtain the following.
\begin{cor}[Uniform Hautus test]
Let $(P,Q,R)$ and $C$ be given matrices.  Then \eqref{eq:our observed system} is observable for all $A\in\calA$ if and only if $\ker C \subseteq \im P$ and for any $\lambda \in\mathbb{C}$ we have
	\begin{equation} \label{eq:hautus obs xyz}
	\rank \begin{pmatrix} R-\lambda QP \\ CP \end{pmatrix} = \rank P.
	\end{equation}  
	Similarly,  \eqref{eq:our observed system} is detectable for all $A\in\calA$ if and only if $\ker C \subseteq \im P$ and \eqref{eq:hautus obs xyz} holds for all $\lambda\in\mathbb{C}$ such that $|\lambda|\geq 1$.
\end{cor}
We now apply the previous result to the situation that input-state-output data on the system are available, as explained in Section II. As before, suppose the data are $(U_-,X,Y_-)$ and consider the affine set $\calA_{\rm dat}$ of all $n \times n$ matrices given by \eqref{eq:compat}. The next result establishes conditions under which the data are informative for observability and for detectability.
\begin{cor} \label{cor:obs and det}
Let $(U_-,X,Y_-)$ be given input-state-output data. Let $(M~ ~N)$ be any matrix such that \eqref{eq:annihilator} holds. Let $R$ be given by \eqref{eq:Z}. The data are informative for observability if and only if 
$\ker C \subseteq \im X_-$ and for all $\lambda \in\mathbb{C}$ we have
\begin{equation}\label{eq:hautus obs data xyz}
	\rank \begin{pmatrix} R - \lambda M X_-  \\ CX_- \end{pmatrix} =  \rank X_- .
\end{equation}  
The data are informative for detectability if and only if $\ker C \subseteq \im X_-$ and \eqref{eq:hautus obs xyz} holds for all  $\lambda \in\mathbb{C}$ with $|\lambda| \geq 1$.
\end{cor}
Again, in the special case that the process noise and measurement noise are independent,  Corollary \ref{cor:obs and det} holds verbatim with $M$ such that $\ker M = \im E$ and $R$ given by \eqref{eq:newZ}. For the case that there is no process noise, in the rank test \eqref{eq:hautus obs data xyz} we should take $M = I_n$ and $R$ given by \eqref{eq:newnewZ}.
\begin{example} \label{ex:ex1}
As an example, consider the system \eqref{eq:full system} with 
\[
A_{\rm true} = \begin{pmatrix} 0 & 1 \\ 2 & 0 \end{pmatrix}, ~B = \begin{pmatrix} 0\\ 1 \end{pmatrix}, ~E = \begin{pmatrix} 1\\ 0 \end{pmatrix}, 
\]
\[
C = \begin{pmatrix} 1 & 0 \end{pmatrix}, ~
D = 0,~ F = 0.
\]
Suppose that the following data are given: 
\begin{equation}  \label{eq:specific data}
U_- = (1 ~ ~1),~ X = \begin{pmatrix} 0 & 0 & 2 \\ 0 & 1 & 1 \end{pmatrix},~ Y_- = ( 0~ ~ 0).
\end{equation}
These data are indeed compatible with the true system, since \eqref{e:data} holds with $W_- = (0~~1)$.
It is easily verified that
\[
\calA_{\rm dat} = \big\{  \begin{pmatrix} a  &  b \\ c & 0 \end{pmatrix} \mid a,b,c \in \mathbb{R} \big\}.
\]
We will check whether the data are informative for strong detectability. Take $M = (0~~1)$. Since $F=0$ we have
$R= M(X_+ - BU_-) = (0~~0)$, $MX_- = (0~~1)$, $CX_- = (0~~0)$, $M B = 1$. The condition $C^{-1} \im D \subseteq \im X_-$ is satisfied, so informativity for strong detectability holds if and only if 
\[
\rank \begin{pmatrix} 0  &  -\lambda & 1 \\ 0  & 0  & 0 \end{pmatrix} = 2
\]
for $|\lambda| \geq 1$, which is clearly not the case. We now check informativity for detectability. This requires $\ker C \subseteq \im X_-$ and 
\[ 
\rank \begin{pmatrix} 0  &  -\lambda \\ 0  & 0   \end{pmatrix} = 1
\]
for $|\lambda| \geq 1$. Both conditions indeed hold. On the other hand, the data are not informative for observability since the rank condition fails for $\lambda = 0$. 
If, in the example, we modify $C$ and take $C =(0~~1)$, and accordingly $Y_- = (0~~1)$, then the data are still not informative for strong observability. In that case the rank condition does hold for all $\lambda \in \mathbb {C}$, but the condition $C^{-1} \im D \subseteq \im X_-$  is violated.
\end{example}
\begin{remark}
For the noiseless case, without proof we mention that if, apart from $A_{\rm true}$, also $C_{\rm true}$ is unknown (but $B$ and $D$ are still known), then both for informativity for observability and detectability a necessary condition  is that $X_-$ has full row rank. As illustrated in Example \ref{ex:ex1}, this is no longer the case if $C_{\rm true}$ is known. Since $X_+= A_{\rm true}X_-+  BU_- $ and $Y_- = C_{\rm true}X_- + DU_-$, this implies $A_{\rm true} = (X_+ -BU_-) X_-^\dagger$ and $C_{\rm true}= (Y_- - DU_-)X_-^\dagger$ for any right-inverse $X_-^\dagger$ of $X_-$. Hence, in that case the data are informative for observability (detectability) if and only if $X_-$ has full row rank, and the pair $((Y_- - DU_-)X_-^\dagger, (X_+ -BU_-) X_-^\dagger)$ is observable (detectable). The unknown $A_{\rm true}$ and $C_{\rm true}$ are then uniquely determined by the data.
\end{remark}

\subsection{Informativity for strong controllability and stabilizability}
For the system \eqref{eq:def sys ABCD}, the dual properties of strong obser\-vability and strong detectability are strong controllability and strong stabilizability. These properties can be defined  in terms of trajectories of the system. Here, for brevity, we define \eqref{eq:def sys ABCD} to be strongly controllable (strongly stabilizable) if the pair $(A +LC, B + LD)$ is controllable (stabilizable) for all $L$. From this it is immediate that \eqref{eq:def sys ABCD} is strongly controllable (strongly stabilizable)  if and only if the dual system
$(A^\top,C^\top,B^\top,D^\top)$ is strongly observable (strongly detectable). As before, assume that $B,C$ and $D$ are given, but that 
$A$ can be any matrix from the affine set $\calA := \{ A \in \mathbb{R}^{n\times n} \mid R = QAP \}$, where $P,Q$ and $R$ are given. Obviously, $A \in \calA$ if and only if $A^\top$ satisfies $R^\top = P^\top A^\top Q^\top$. 
The above observations make the following a matter of course.
\begin{cor} \label{thm:str cont hautus}
	Let $(P,Q,R)$ and $(B,C,D)$ be given. Then \eqref{eq:def sys ABCD} is strongly controllable for all $A\in\calA$ if and only if $\ker Q \subseteq B\ker D$ and for all $\lambda \in\mathbb{C}$
	\begin{equation}\label{eq:hautus str cont xyz}
	\rank\begin{pmatrix} R-\lambda QP & QB\\ CP &D \end{pmatrix} = \rank Q + \rank \begin{pmatrix} CP  & D \end{pmatrix}.
	\end{equation}  
	Similarly,  \eqref{eq:def sys ABCD} is strongly stabilizable for all $A\in\calA$ if and only if $\ker Q \subseteq B\ker D$ and \eqref{eq:hautus str cont xyz} holds for all $\lambda\in\mathbb{C}$ such that $|\lambda|\geq 1$. 	
\end{cor}
%
Since a given pair $(A,B)$ is controllable (stabilizable) if and only if the quadruple $(A,B,0,0)$ is strongly controllable (strongly stabilizable), the following also follows immediately.
\begin{cor}[Uniform Hautus test] \label{cor:haut cont}
	Given $(P,Q,R)$ and $B$, the pair $(A,B)$ is controllable for all $A\in\calA$ if and only if $\ker Q \subseteq \im B$ and for any $\lambda \in\mathbb{C}$ 
	\begin{equation}\label{eq:hautus cont xyz}
	\rank\begin{pmatrix} R-\lambda QP & QB\end{pmatrix} = \rank Q.
	\end{equation}  
	Furthermore $(A,B)$ is stabilizable if and only if $\ker Q \subseteq \im B$ and  \eqref{eq:hautus cont xyz} holds for all $\lambda\in\mathbb{C}$ such that $|\lambda|\geq 1$. 
\end{cor}
By appying the above in the context of informativity, we immediately obtain the following.
\begin{cor} \label{cor:strong cont and stab}
Let $(M~~N)$ be such that \eqref{eq:annihilator} holds. Given the data $(U_-,X,Y_-)$, let $R$ be given by \eqref{eq:Z}. The data  are informative for strong controllability if and only if \/ $\ker M \subseteq \im B$  and for all $\lambda \in\mathbb{C}$ we have
\begin{equation}\label{eq:hautus str cont data xyz}
	\rank\begin{pmatrix} R -\lambda MX_- & M B\\ CX_- & D \end{pmatrix} = \rank M + \rank (CX_-  ~~ D ).
\end{equation} 
The data are informative for strong stabilizability if and only if  $\ker M \subseteq \im B$ and \eqref{eq:hautus str cont data xyz} holds for all  $\lambda \in\mathbb{C}$ with $|\lambda| \geq 1$.
\end{cor}
\begin{cor} \label{cor:cont and stab}
Let $(M~~N)$ be such that \eqref{eq:annihilator} holds and let $R$ be given by \eqref{eq:Z}. The data $(U_-,X,Y_-)$ are informative for controllability if and only if \/ $\ker M \subseteq \im B$  and for all $\lambda \in\mathbb{C}$ we have
\begin{equation}\label{eq:hautus cont data xyz}
	\rank\begin{pmatrix} R-\lambda MX_- & M B \end{pmatrix} = \rank M.
\end{equation} 
The data are informative for stabilizability if and only if  $\ker M \subseteq \im B$ and \eqref{eq:hautus cont data xyz} holds for all  $\lambda \in\mathbb{C}$ with $|\lambda| \geq 1$.
\end{cor}
As before, in the special case of independent process and measurement noise,  Corollary \ref{cor:strong cont and stab} and \ref{cor:cont and stab} hold verbatim with $M$ such that $\ker M = \im E$ and $R$ given by \eqref{eq:newZ}. In this special case, the rank test for controllability and stabilizability can be simplified to 
$\rank M \begin{pmatrix} X_+ - \lambda X_- & B \end{pmatrix} = \rank M$
for all $\lambda \in \mathbb{C}$, and $\lambda \in \mathbb{C}$ with $|\lambda| \geq 1$, respectively.

If there is no process noise, in the rank tests \eqref{eq:hautus str cont data xyz} and \eqref{eq:hautus cont data xyz} we should take $M = I_n$ and $R=X_+ - B U_-$.  For this special case, the rank test for controllability and stabilizability can even be simplified to 
\begin{equation} \label{eq:even simpler}
\rank \begin{pmatrix} X_+ - \lambda X_- & B \end{pmatrix} = n
\end{equation}
for all $\lambda \in \mathbb{C}$, and  $\lambda \in \mathbb{C}$ with $|\lambda| \geq 1$, respectively.

\begin{remark}

	The rank test \eqref{eq:even simpler} can also be derived from \cite[Theorem 8]{vanWaarde2020}. Indeed,  that theorem states that all pairs $(A,B)$ that satisfy the linear equation $X_+ = AX_- +BU_-$ are controllable if and only if $\rank \begin{pmatrix} X_+-\lambda X_-\end{pmatrix}=n$ for all $\lambda\in\mathbb{C}$. 
	This result can be applied to our set up, where we assume that only $A$ is unknown and that $B$ is given. Indeed, by defining `new data' by
	\begin{equation} 	\label{eq:B known} 
	\tilde{X}_+:= \begin{bmatrix} X_+ & B \end{bmatrix}, \tilde{X}_-: = \begin{bmatrix} X_- & 0 \end{bmatrix}, \tilde{U}_-:= \begin{bmatrix} U_- & I_m  \end{bmatrix},
	\end{equation} 
we have that a matrix $A$ satisfies $X_+ = AX_- +BU_-$  if and only if $(A,B)$ satisfies $\tilde{X}_+ = A\tilde{X}_- +B\tilde{U}_-$. By applying\cite[Theorem 8]{vanWaarde2020} to the new data \eqref{eq:B known} we then get that $(A,B)$ is controllable for all $A$ satisfying  $X_+ = AX_- +BU_-$ if and only if \eqref{eq:even simpler} holds.
	\end{remark} 
\begin{example}
Again take as the true system the one specified in Example \ref{ex:ex1}. Also, let the data be given by 
\eqref{eq:specific data}. Note that the condition $\ker M \subset \im B$ is violated, so the data are neither informative for strong controllability nor for strong stabilizability. They are also not informative for controllability or stabilizability.
\end{example}

%
\section{A geometric approach to informativity} \label{sec:geometric}
It is well known, see for example \cite{Trentelman2001}, that observability and strong observability also allow tests in terms of certain subspaces of the state space, more specifically, the unobser\-vable subspace and weakly unobservable subspace. Properties of the weakly unobservable subspace also characterize left-invertibility of the system. In this section we will use these ideas to characterize informativity for strong observability, observability and left-invertibility. 

Again consider the system \eqref{eq:def sys ABCD}. We call a subspace $\calV \subseteq \mathbb{R}^n$ output-nulling controlled invariant  if
\begin{equation}  \label{eq:output nulling}
\begin{bmatrix} A \\ C \end{bmatrix} \calV \subseteq \calV \times \{0\} +\im \begin{bmatrix} B\\ D\end{bmatrix},
\end{equation}
(see \cite{Molinari1976,Trentelman2001}). Since any finite sum of such subspaces retains this property, there exists a unique largest output-nulling controlled invariant subspace, which will be denoted by  $\calV(A,B,C,D)$. This subspace is called the {\em weakly unobservable subspace} of the system \eqref{eq:def sys ABCD}. 
The system \eqref{eq:def sys ABCD} is strongly observable if and only if $\calV(A,B,C,D) = \{0\}$, see  \cite[pp. 159-160 and Theorem 7.16]{Trentelman2001}.

Now, again consider the situation that the matrices $B,C$ and $D$ are specified, but that $A$ can be any matrix from the affine set $\calA$ given by \eqref{eq:def sig a}, where the matrices $P \in \mathbb{R}^{n \times r}$, $Q\in \mathbb{R}^{\ell \times n}$ and $R\in \mathbb{R}^{\ell \times r}$ are given. We consider the set of all subspaces $\calJ \subseteq \mathbb{R}^r$ that satisfy the following inclusion:
\begin{equation}\label{eq:def Jstar}
\begin{bmatrix} R \\ CP \end{bmatrix} \calJ \subseteq QP\calJ  \times \{0\} +\im \begin{bmatrix} QB\\ D\end{bmatrix}.
\end{equation}
It is easily verified that any finite sum retains this property, and therefore there exists a largest subspace of $\mathbb{R}^r$ that satisfies the inclusion \eqref{eq:def Jstar}. We will denote this subspace by $\calJ^\star$.
\begin{remark}\label{rem: iterating Jstar}
	It is straightforward to check that $\calJ^\star$ can be found from $B,C,D, P,Q$ and $R$ in at most $r$ steps by letting $\calJ_0=\mathbb{R}^r$, and iterating
	\begin{equation}  \label{eq:iteration}
	\calJ_{t+1} = \begin{bmatrix} R \\ CP \end{bmatrix}^{-1}\left( QP\calJ_t  \times \{0\} +\im \begin{bmatrix} QB\\ D\end{bmatrix}\right).
	\end{equation}
\end{remark}
The following result will be instrumental in the remainder of this section.
\begin{theorem}\label{thm: Vstar and Jstar}
Let $(P,Q,R)$ and $(B,C,D)$ be such that $C^{-1}\im D \subseteq \im P$. Then the following hold: 
\begin{enumerate}
\item\label{item: theorem spaces 1} 
For all $A\in\calA$, we have $\calV(A,B,C,D)\subseteq P\calJ^\star$. 
\item\label{item: theorem spaces 2} 
There exists $\bar{A}\in\calA$ such that $P\calJ^\star \subseteq \calV(\bar{A},B,C,D)$.
\end{enumerate} 
\end{theorem} 
\begin{proof}
 \eqref{item: theorem spaces 1}: Assume that $C^{-1}\im D \subseteq \im P $ holds. Let $A\in \calA$ and let $\calV\subseteq\mathbb{R}^n$ be an output nulling controlled invariant subspace. 
Note that $C\calV \subseteq\im D$, and therefore there exists a subspace $\calJ$ such that $\calV=P\calJ$. We now see that
\[
\begin{bmatrix} R \\ CP \end{bmatrix} \calJ = \begin{pmatrix} Q& 0\\ 0& I \end{pmatrix}\begin{bmatrix} A\\ C\end{bmatrix} P\calJ \subseteq QP\calJ  \times \{0\} +\im \begin{bmatrix} QB\\ D\end{bmatrix}.
\]
Due to the definition of $\calJ^\star$, we then obtain $\calV(A,B,C,D) \subseteq P\calJ^\star$. 

\eqref{item: theorem spaces 2}: Let $\calJ$ satisfy \eqref{eq:def Jstar}. Then for any $A\in\calA$ and $x \in P\calJ$ there exists $u \in \mathbb{R}^m$ such that:
\[ Cx+Du =0, \qand QAx +QBu \in QP\calJ. \]
This implies that 
\[ 
Ax+Bu \subseteq Q^{-1}Q(Ax+Bu) \subseteq Q^{-1}QP\calJ = P\calJ + \ker Q. 
\]
Now let $\{x_1,...,x_k\}$ be a basis of the subspace $P\calJ$. We can use the previous to write $Ax_i+Bu_i = y_i+z_i$, where $Cx_i+Du_i=0$, $y_i\in P\calJ$ and $z_i\in \ker Q$. Let $A_0$ be any real $n \times n$ matrix such that $A_0x_i= -z_i$ for $i=1,\ldots,k$ and $QA_0P=0$. Then, if we define $\bar{A}=A+A_0$, we see that $\bar{A}\in \calA$. By definition $\bar{A}x_i+Bu_i =y_i \in P\calJ$, and therefore, if we write $\calV = P\calJ$, we have: 
\[
\begin{bmatrix} \bar{A} \\ C \end{bmatrix} \calV \subseteq \calV \times \{0\} +\im \begin{bmatrix} B\\ D\end{bmatrix}.
\]
Therefore $P\calJ\subseteq \calV(\bar{A},B,C,D)$, proving that $P\calJ^\star\subseteq \calV(\bar{A},B,C,D)$.  
\end{proof} 

Using Theorem \ref{thm: Vstar and Jstar} we immediately obtain the following.
\begin{theorem}\label{thm:str obs spaces}
	Let $(B,C,D)$ and $(P,Q,R)$ be given. Then the system \eqref{eq:def sys ABCD} is strongly observable for all $A\in\calA$ if and only if $C^{-1}\im D\subseteq \im P$ and $\calJ^\star \subseteq \ker P$.
\end{theorem}
\begin{proof} From Theorem~\ref{thm:str obs hautus} we see that $C^{-1}\im D\subseteq \im P$ is a necessary condition. The rest follows from Theorem~\ref{thm: Vstar and Jstar}.
\end{proof}
The procedure can be mimicked in order to characterize observability. For the system \eqref{eq:our observed system}, the unobservable subspace $\calN$ is the largest $A$-invariant subspace contained in $\ker C$, and \eqref{eq:our observed system} is observable if and only if $\calN = \{0\}$. In the situation that only $C$  and matrices $(P,Q,R)$ are given, while $A$ can be any matrix in the affine set $\calA$, we should look at the largest subspace $\calL \subseteq \mathbb{R}^r$ with the properties that 
\begin{equation} \label{eq:unobservable subspace}
R\calL \subseteq QP \calL \mbox{ and } CP\calL = \{0\}.
\end{equation}
Denote this subspace by $\calL^\star$ Then we have 
\begin{cor}\label{cor:obs spaces}
	Given $(P,Q,R)$ and $C$, then \eqref{eq:our observed system} is observable for all $A\in\calA$ if and only if $\ker C\subseteq \im P$ and $\calL^{\star} \subseteq \ker P$.
\end{cor}
The subspace $\calL^\star$ is obtained in at most $r$ steps by applying the iteration \eqref{eq:iteration} with $B =0$ and $D = 0$.

We now very briefly put the above in the context of informativity of input-state-output data. As before, let 
$(U_-,X, Y_-)$ be the noisy data obtained from the system \eqref{eq:full system}. Let $(M~~N)$ be any matrix such that \eqref{eq:annihilator} holds. Then, by Theorem \ref{thm: Vstar and Jstar}, these data are informative for strong observability of \eqref{eq:def sys ABCD} if and only if $C^{-1}\im D \subseteq \im X_-$ and $\calJ^\star \subseteq \ker X_-$, where $\calJ^\star$ is the largest subspace satisfying \eqref{eq:def Jstar} with $R$ given by \eqref{eq:Z}, $P = X_-$ and $Q = M$. 
Likewise, informativity for observability holds if and only if $\ker C \subseteq \im X_-$ and $\calL^\star \subseteq \ker X_-$.

Obviously, the above can, again, be dualized to obtain alternative tests for informativity for controllability and strong controllability. We omit the details here. Instead, we will turn to informativity for the property of left-invertibility of the system \eqref{eq:def sys ABCD} now. We briefly recall the definition.
\begin{definition}
The system \eqref{eq:def sys ABCD} is called \textit{left-invertible} if for each input sequence $\bmu$ the following holds: $\bmy(t, 0, \bmu)=0$ for all $t \in \mathbb{Z}_+$ implies that $\bmu(t) =0$ for all $t \in \mathbb{Z}_+$. 
\end{definition}

The following characterization of left-invertibility was given in \cite[Thm. 8.26]{Trentelman2001}.
\begin{prop} 
	The following are equivalent: 
	\begin{enumerate} 
		\item The system \eqref{eq:def sys ABCD} is left-invertible.
		\item  $\calV(A,B,C,D)\cap B\ker D = \{0\}$ and $\begin{pmatrix} B \\D \end{pmatrix}$ has full column rank.
	\end{enumerate}
\end{prop} 
\noindent The next result then, again, follows from Theorem \ref{thm: Vstar and Jstar}.
\begin{theorem}\label{thm:l-inv spaces}
	Let $(P,Q,R)$ and $(B,C,D)$ be given. Assume that $C^{-1} \im D \subseteq \im P$. Then the system \eqref{eq:def sys ABCD} is left-invertible for all $A\in\calA$ if and only if $P\calJ^\star \cap B\ker D =\{0\}$
	and $\begin{pmatrix} B \\D \end{pmatrix}$ has full column rank.
\end{theorem}
As before, this can immediately be applied in the context of informativity. We omit the details.
\begin{remark}
	Note that Theorem \ref{thm:l-inv spaces} requires $C^{-1} \im D \subseteq \im P$, which, unfortunately, for left-invertibility for all $A \in \calA$ is not a necessary condition. This can be seen, for example, by taking $D=I$. Then, regardless of our choice of $(P,Q,R)$, $B$ and $C$, we see that \eqref{eq:def sys ABCD} is left-invertible for all $A\in\calA$. However, in this case $C^{-1}\im D = \mathbb{R}^n$, so the condition  $C^{-1} \im D \subseteq \im P$ is violated if $P$ does not have full row rank.
\end{remark}
To conclude this section, we note that Theorem \ref{thm:l-inv spaces} can be dualized in a straightforward way to obtain a characterization of right-invertibility for all $A\in\calA$, and conditions for informativity of data for right-invertibility. Again, we omit the details. 

To illustrate the the theory developed in this section we give the following example.

\begin{example} 
Consider the system \eqref{eq:full system} with
\[
A_{\textrm{true}} = \begin{bmatrix} 0 & 1 & 0 & 0 \\ 0 & 0 &1  & 0 \\ 0 & 0& 0& 1 \\ 0&0& 0&0 \end{bmatrix} , \quad B= \begin{bmatrix} 1 \\ 0 \\0 \\0 \end{bmatrix},\quad E= \begin{bmatrix} 0 \\ 0 \\0 \\1 \end{bmatrix}, \]
\[C= \begin{bmatrix} 1 & 0& 0& 0 \end{bmatrix}, \quad D = 0,\quad F=0.
\] 
Let data be given by
\[ X= \begin{bmatrix} 0&0&0&5 \\ 0&0&1&0\\0&1&0&0\\1&0&0&0\end{bmatrix}, \;U_-=\begin{bmatrix} 0 &0&4\end{bmatrix}, \; Y_- = \begin{bmatrix} 0 &0&0 \end{bmatrix}. \] 
Since there is only process noise, we should take $M$ such that $\ker M = \im E$. Define
\[ 
M :=  \begin{bmatrix} 1 & 0&0&0 \\ 0&1&0&0 \\0&0&1&0 \\0&0&0&0 \end{bmatrix}. 
\]
Then
\[
R = M (X_+ - BU_-) = \begin{bmatrix} 0&0&1\\0&1&0\\1&0&0 \\ 0 & 0 & 0 \end{bmatrix}.
\]
It is easily verified that 
\[ 
\calA_{\rm dat} = \left\lbrace \begin{bmatrix} a_{11} &1 & 0& 0 \\ a_{21} & 0 & 1 &0 \\ a_{31} &0&0&1 \\ a_{41} & a_{42} & a_{43} & a_{44} \end{bmatrix} \mid a_{ij} \in \mathbb{R} \right\rbrace .
\]
Note that $C^{-1}\im D \subseteq \im X_-$. In this case $\calJ^\star = \mathbb{R}^3$, and therefore the data are not informative for strong observability.
On the other hand, $\calL^\star=\zset$, proving that we do have informativity for observability. 

If we modify our system by taking $B = e_i$, the $i$th standard basis vector in $\mathbb{R}^4$ ($i = 2,3,4$), and adapt the data $X$ accordingly, we get 
$\calJ^\star=\mathbb{R}^{4-i}\times \{0\}^{i-1}$. This means that $X_- \calJ^\star= \{0\}^i\times \mathbb{R}^{4-i}$. Thus, only for $i=4$, the data are informative for strong observability. For $i = 2,3,4$ the data are informative for left-invertibility.

\end{example}

\section{Conclusions} \label{sec:conclusions}
In this paper we have given necessary and sufficient conditions for informativity of noisy data obtained from a given unknown system for a range of system properties. These conditions are in terms of rank tests on polynomial matrices that can be constructed from these noisy data. The main instrument used to obtain these tests was a general theorem that expresses a rank property of the Rosenbrock system matrix of an unknown system in terms of a polynomial matrix that collects available information about that system. We have also set up a geometric framework for informativity analysis. Within that framework we have found geometric tests for informativity of data for strong observability, observability, and left-invertibility. 

Within the framework of this paper, no assumptions are made on the noise samples, and in that sense our noise model is very general. On the other hand, complete knowledge on how the noise influences the system is assumed to be available (via $E$ and $F$). A drawback of this noise model is that in some situations it may not be possible to draw conclusions on the system on the basis of data. For example, if within our framework $E =I$  and $F=0$, it is impossible to draw conclusions from data, no matter how many input/state/output samples have been collected. An interesting problem for future research would therefore be to investigate data-driven analysis from noisy data under the assumption that the noise samples $W_-$ are bounded. Relevant noise models with bounded noise samples have, for example, been proposed in \cite{DePersis2020,Berberich2020,Koch2020,vanWaarde2020b}.

As an extension of the work in this paper we also see its generalization to the case that, apart from the $A$-matrix of the unknown system, also (parts of) the matrices $B$, $C$ and $D$ are unknown. Finally, it would be interesting to apply Theorem \ref{thm:rank equivalences} in the context of structured systems, where specific entries of the system matrices are constrained to satisfy certain linear equations, and the remaining entries are arbitrary. 
%
%

\bibliographystyle{plain}
\bibliography{references.bib}

 \end{document}